\documentclass{amsart}

\usepackage{amssymb}
\usepackage{amsthm}
\usepackage{amsmath}
\usepackage[usenames]{color}
\numberwithin{equation}{section}

\title{On the continuation of locally operator monotone functions}
\author{
J. E. Pascoe
}

\def\rtwo{\mathbb{R}^2}

\def\by-half{b_Y^{-\frac{1}{2}}}

\def\rplus{\mathbb{R}^+}
\def\rtwoplus{{(\mathbb{R}^+)}^2}
\def\rnplus{{(\mathbb{R}^+)}^n}
\def\be{\begin{equation}}
\def\ee{\end{equation}}

\def\mx{\text{max}}

\def\be{\begin{equation}}
\def\ee{\end{equation}}

\def\ni#1{\|#1\|_\infty}
\def\L{\mathcal{L}}

\def\mxthm{\text{\emph{max}}}

\def\rn{\mathbb{R}^n}

\def\cn{\mathbb{C}^n}

\def\p{\mathcal{P}}

\def\res{{(A-z_Y)}^{-1}}

\def\Dx{\delta}
\def\x{\vec{x}}
\def\z{\vec{z}}

\def\Qh{\hat{Q}^n_{p}(\Dx;m_1,m_2)}
\def\Qhtwo{\hat{Q}^2_{p}(\Dx;m_1,m_2)}
\def\Q{Q^n_{p}(\Dx;m_1,m_2)}

\def\W{W^n_{p}(\Dx;m_1,m_2)}

\def\q#1{q^n_{#1}(\Dx;m_1,m_2)}
\def\qtwo#1{q^2_{#1}(\Dx;m_1,m_2)}

\def\McCarthy{M\raise.45ex\hbox{c}Carthy }
\def\McCarthyc{M\raise.45ex\hbox{c}Carthy, }

\begin{document}

\bibliographystyle{plain}
\subjclass[2010]{Primary 47A63 ; Secondary 32A40 }

\newtheorem{defin}[equation]{Definition}
\newtheorem{lem}[equation]{Lemma}
\newtheorem{prop}[equation]{Proposition}
\newtheorem{thm}[equation]{Theorem}
\newtheorem{claim}[equation]{Claim}
\newtheorem{cor}[equation]{Corollary}
\newtheorem{ques}[equation]{Question}
\begin{abstract}
	We generalize the phenomenon of continuation from complex analysis to locally operator monotone functions.
	Along the lines of the egde-of-the-wedge theorem,
	we prove continuations exist dependent only on geometric features of the domain and, namely, independent of the function values.
	We prove a generalization of the Julia inequality for a class of functions containing locally operator monotone functions, Pick functions.
\end{abstract}
\maketitle

\section{Introduction}
	\subsection{Overview}
		A real-valued function of one variable, $f,$ defined on an open interval $I\subseteq \mathbb{R}$ is said to be \emph{operator monotone} if
		for any two self adjoint operators $A$ and $B$ whose spectra are completely contained in $I,$
		$$A\leq B \Rightarrow f(A) \leq f(B).$$
		It was shown in \cite{lo34}  by L\"owner in 1934 that if $f$ is operator monotone, then $f$ extends to an analytic function that maps the upper half plane to itself.
		In \cite{amy11c}, Agler, \McCarthy and Young extended this theorem to several variables. In fact, they generalized L\"owner's theorem to a possibly larger class,
		the \emph{locally operator monotone functions.} That is, if $f$ which takes an open set $U\subset \mathbb{R}^n$ into $\mathbb{R}$ is locally operator monotone, (that is,
		$$S'(t)\geq 0 \Rightarrow \frac{d}{dt}f(S(t))\geq 0$$
		for every $C^1$ parametrized curve $S(t)$ in the set of commuting $n$-tuples of operators with $\sigma(S(t))\subset U$)
		then $f$ extends to an analytic map on a neighborhood of $U$ in $\mathbb{C}^n$ with certain special properties similar to those in L\"owner's theorem.

		The implication of analyticity involved in these results is peculiar, and, as we will show, powerful.
		Namely, for our purposes, it will imply that locally operator monotone functions enjoy the
		phenomenon of analytic continuation.

		Inspired by the edge-of-the-wedge theorem\cite{rudeow},
		our main result, a \emph{wedge-of-the-edge theorem}, endeavors to describe the set on which a locally operator monotone function continues soley in terms
		of its domain. We show that, for a given open set $U\subset \mathbb{R}^n,$ there is a well-described set $\hat{U}\supseteq U$
		such that every locally operator monotone function on $U$ extends to $\hat{U}.$
		For example, our theory implies that if $U = \{(x,y)\in\rtwo | xy >  -1\},$ then any locally operator monotone function on $U$ extends to all of $\rtwo.$
		In general, however, the size of the continuation can not be guaranteed to be so large, and is described by subtle formulas that we call \emph{regulators.}

		We use a recently developed tool in the calculus of Pick Functions at infinity in two variables
		(based of the classical work of R. Nevanlinna in \cite{nev22}) descibed in a paper of Agler, Tully-Doyle and Young \cite{aty11} and
		a paper of Agler and \McCarthy \cite{am11}.
		
		We develop a modification of the Julia inequality in two or more variables for the polydisk,
		which was studied in Abate \cite{ab98}, \cite{ab05}, Agler, \McCarthyc Young \cite{amy11a}, and Jafari \cite{jaf93}, to develop a theory of
		bounds on the directional derivatives of locally operator monotone functions in terms of the domain. This is the content of
		our \emph{local Julia inequalities.} These give some quantative insight into the behavior of locally operator monotone functions.
		This theory, by itself, will allow us to conclude that all entire locally operator monotone functions are linear.
	
	\subsection{Main results}
		We define locally operator monotone functions.
		\begin{defin}
			Let $U \subseteq \mathbb{R}^n$ be open.
			Let $f:U\rightarrow \mathbb{R}$ be a $C^1$.
			The function $f$ is said to be \emph{locally operator monotone} on $U$ if
				$$S'(t)\geq 0 \Rightarrow \frac{d}{dt}f(S(t))\geq 0$$
			for every $C^1$ parametrized curve $S(t)$ in the set of commuting $n$-tuples of operators with $\sigma(S(t))\subset U.$
		
			We denote the class of all locally operator monotone functions on $U$ as $L(U).$
			For a general set $E,$ we set $L(E) = L(\text{\emph{int} }E).$
		\end{defin}
	
		The following proposition establishes the theory of continuation for locally operator monotone functions.
		\begin{prop}\label{maxcontexist}
			Let $U\subset \mathbb{R}^n$ be open and nonempty. For any locally operator monotone function $f$ there is a unique maximal
			$\hat{U}$ and $F \in L(\hat{U})$ such that $F|_U = f.$
		\end{prop}
		That is, each locally operator monotone function has a unique maximal continuation.
		We prove this in Section \ref{preliminaries}.
		Computing the $\hat{U}$ and $F$ for a certain $f$ in detail depends on the function $f$ itself. However, this expostion will not attempt to create a machinery
		for individual functions. Instead, in the vein of the edge-of-the-wedge theorem, we give conditions dependent only on the domain $U$ that determine the domain
		of extension $\hat{U}.$ That is, we examine the problem of when $L(U)=L(\hat{U}).$
	
		A version edge-of-the-wedge theorem gives that all of the functions that are analytic on some cone in $\mathbb{C}^n$ and on the
		negative of that cone and are continuous on some relatively open set $S$ in their shared boundary must analytically continue some fixed to a neighborhood of $S.$
		For a detailed account, see Rudin's book dedicated to the theorem, \cite{rudeow}.
		The following definition gives the description of the sets on which we intend to analytically continue in our wedge-of-the-edge theorem. 
		\begin{defin}[Regulators] 
			Let $S\subset \mathbb{R}^n.$
			\begin{enumerate}
				\item At a point $p\in \rn$ denote 
					$$d^S(p)[\x]= \text{\emph{min}}(\sup\{s|p+s\x\in S\},\sup\{s|a-s\x\in S\}).$$
				\item At each point $p\in S$ define the \emph{local regulator}
					$$q^S_{k}(p)[z] = \sup\{|q(z)||q\in\text{\emph{Proj}}(\mathbb{R}[\x]),$$
					$$ \text{\emph{deg }}q=k, q(d^S(p)[\x]) \leq d(a)[\x],
					\forall \x\in \rnplus \}.$$
				\item The \emph{locally regulated set} at $a$ is
					$$\hat{Q}^S(p) = \text{\emph{int }}\{p+z\in\cn|\limsup_{k\rightarrow\infty} \sqrt[k]{q^S_{k}(p)[z]} < 1\}.$$
				\item The \emph{real locally regulated set} for $S$ at $p$ is
					$$Q^S(p) = \hat{Q}^S(p) \cap \rn.$$
				\item The \emph{real regulated set} for $S$ is the minimal set $Q^S \supset S$ such that
					for every $p$ in the interior of $Q^S$
					$$Q^S \supset Q^{Q^S}(p).$$
			\end{enumerate}
		\end{defin}
		
		Now, with this language, we can state our main result.	It describes continuations in terms of the above regulators.
		\begin{thm}[A wedge-of-the-edge theorem]\label{mainwoe}
			Let $U\subset \rn$ be open.
				$$L(U)=L(Q^U)$$
		\end{thm}
		
		The regulators may seem a bit mysterious, but they do have some qualitative properties.
		Note, for example, that $q^S_{k}(p)[z]$ depends only on the size of the domain in the totally positive directions and totally negative directions and becomes smaller
		for larger sets. Thus, the size of $Q^S(p)$ depends only on the size of $S$ in those directions with respect to $p$ and and will become bigger when $S$ is extends further
		in those directions. This then translates into similar properties for how $Q^S$ relates to $S.$ 
		
		In terms of continuation, this qualifies that the size of the domain of the continuation of
		a locally operator monotone function is determined by the size of its domain in the totally positive and negative directions.
		An extreme example of this is the example given in the overview: If $U = \{(x,y)\in\rtwo | xy >  -1\},$ then any locally operator monotone function on $U$ extends to all of $\rtwo.$
		
		On the other hand, the effect of the indefinite directions, those with both positive and negative components, seem to have little or no effect.
		This can be seen in simple functions such as 
			$$\frac{x}{1-xy}$$
		which is now known to be locally operator monotone via the main result in \cite{amy11c},
		which is stated in our preliminary section as Theorem \ref{amy11c}. This function is singular on the hyperbola $xy = 1,$
		and cannot be extended there.
		
		In Section \ref{regest}, we show that the regulators are nontrival and give some of their basic geometry.
		Specifically, we give estimates on their sizes for certain wedge-shaped sets.
		
		In Section \ref{juliasection}, we generalize the Julia inequality to obtain our local Julia inequalities.
		This section essentially gives bounds on the derivatives of locally operator monotone functions
		in terms of the geometry of a domain. We use this to show that a series representation of a locally operator
		monotone converges somewhere, and, thus, that the function must continue there in Section \ref{woesection}.
		
		For certain domains the theory of local Julia inequalities makes the computation of $L(U)$ tractible. In fact, we give a complete description of $L(\mathbb{R}^n).$
		This can be found in one variable in \cite{bh97}.
		\begin{prop}\label{liouville}
			All entire locally operator monotone functions are linear. That is,
				$$L(\mathbb{R}^n) = \{f| f(w) = \x\cdot w, \x \in (\mathbb{R}^{\geq 0})^n\}.$$
		\end{prop}
		
		%

\section{Preliminaries} \label{preliminaries}
	Throughout this expostion we will denote the upper half plane as
		$$\Pi = \{z| \text{Im }z>0\}.$$

	\begin{defin}
		Let $E \subseteq \mathbb{R}^n.$
		The Pick class at $E$, $\mathcal{P}(E),$ is the set of all functions $f:\Pi^n \cup E\rightarrow \overline{\Pi}$ such that $f$ is analytic on $\Pi^n$ and continuous on
		$\Pi^n \cup E.$
	\end{defin}
	We define a subclass of Pick functions related to the locally operator monotone functions, the L\"owner class as in \cite{amy11c}.
	\begin{defin}
		Let $E \subseteq \mathbb{R}^n.$
		The L\"owner class at $E$, $\mathcal{L}(E),$ is the set of all functions $f\in \p(E)$ such that there exist d positive semidefinite
		kernel fuctions $A^i$ such that
			$f(z)-\overline{f(w)}= \sum^n_{i=1} (z_i - w_i)A^i(z,w).$
	\end{defin}
	This class differs slightly from the Pick class, however the are the same if $n=1, 2.$
	We only use that the fact that $\L(E)\subset\p(E).$ The inclusion of the L\"owner class in the Pick class conformally mirrors that between the
	\emph{Schur} and \emph{Schur-Agler} classes in the study of polydisks, and determines the applicability of Hilbert space methods.\cite{ag90}
	
	The following aforementioned theorem was proven in \cite{amy11c}. We present it in our language.
	\begin{thm}[Agler, \McCarthyc Young] \label{amy11c}
		Suppose $U\subseteq \mathbb{R}^n$ is an open set.
		A function $f$ is in the class $L(U)$ if and only if there is a function $F$ such that $F|_U = f$ and $F$ is in the class $\L(U).$
	\end{thm}
	
	We now prove a proposition that allows us to identify $f$ with its extension uniquely.
	\begin{prop}
		Suppose $U\subseteq \mathbb{R}^n$ is an open set.
		Let  $f \in L(U).$
		If two functions $F,G \in \L(U)$ satisfy $F|_U = f$ and $G|_U = f,$ then $F = G.$
	\end{prop}
	\begin{proof}
		By the Schwarz reflection principle, $F$ and  $G$ extend to the set $-\Pi^n.$ Now applying the edge of the wedge theorem, there is $D,$
		a neighborhood of $U$ in $\mathbb{C}^n$ such that $F$ and $G$ extend to $D.$ Now, $F-G\equiv 0$ on $U$. Open sets in
		$\mathbb{R}^n$ are sets of uniqueness for analytic functions, so $F-G\equiv 0$ on $D.$ Thus since $F$ and $G$
		agree on an open set, they must be equal by analytic continuation.
	\end{proof}
	Thus, for locally operator monotone functions, we abuse notation to identify $f$ with its extension. (Such an extension must be unique because,
	by the edge-of-the-wedge theorem, $f$ analytically continues to a neighborhood of $U$ in $\mathbb{C}^n$.) So in the case of open sets we can
	also make the identification $\L(U)=L(U).$
	
	We now prove \ref{maxcontexist} as corollary of \ref{amy11c}.
	\begin{proof}[Proof of \ref{maxcontexist}]
		Suppose $f\in L(U).$ Then $f$ extends to a function on $\Pi^n.$ Let $\hat{U}$ be the union
		of all open sets $V\subseteq \mathbb{R}^n$ such that $f$ extends continuously to $V$ as $f_V$.
		Note, for any two such sets $W, V,$ we get that their extensions agree on $W\cap V$
		because they are both equal to the limit of the values from the function $f$.
		Thus, $f$ continues to $\hat{U}$ as some $F$ which is maximal and unique by its definition.
	\end{proof}
	
	\subsection{The Nevanlinna representation}
		We use the following representation
		due to R. Nevanlinna \cite{nev22}, presented in the following form.
		\begin{thm}[Nevanlinna] 
			Let $h:\Pi\rightarrow \Pi$ is analytic and satisfies
				$$\limsup_{s\rightarrow\infty} s|h(is)|<\infty.$$
			Then, there is a unique finite Borel measure $\mu$ on $\mathbb{R}$ such that
				$$h(z)=\int \frac{1}{t-z}d\mu(t).$$
		\end{thm}
		
		We now use this representation to prove a local Julia inequality for one variable Pick functions.
		This will be essential to the proof of the several variables analogue given in Theorem \ref{julialine}.
		\begin{lem} \label{juliaone}
			Suppose $h:\Pi\rightarrow \Pi$ is analytic, extends to a continuous
			real valued function on $(-b,b)$ and $h(0)=0$.
			Then, $h$ is analytic on $(-b,b).$
			The power series expansion of $h(bz)$ at zero
				$h(bz) = \sum^{\infty}_{n=1} a_n z^n$
			satisfies
				$|a_k| \leq a_1.$	
		\end{lem}
		\begin{proof}
			The analyticity of $h$ is classical and comes from the Schwarz Reflection Principle. We
			leave the details of this calculation to the reader.
			
			Let $\hat{h}(z) = h(\frac{b}{z}).$ This has the Laurent expansion,
			$\hat{h}(z) = \sum^{\infty}_{k=1} a_k z^{-k}.$
			Apply the Nevanlinna representation to $\hat{h}.$
				$$\hat{h}(z)=\int \frac{1}{t-z}d\mu(t).$$
				
			We now note $\mu(\mathbb{R}\setminus [-1,1]) = 0.$ This is an exercise in measure theory. For similar manipulations, see Bhatia \cite{bh97}.
			
			
			Thus, everything about
				$$\hat{h}(z)=\int \frac{1}{t-z}d\mu(t) =\int \sum^{\infty}_{n=0} \frac{t^n}{z^{n+1}} d\mu(t)$$
			is absolutely and uniformly convergent on $|z| > 1+\epsilon.$ So we can interchange
			the summation and integral.
				$$\hat{h}(z)=\sum^{\infty}_{n=0} \frac{1}{z^{n+1}}\int t^n d\mu(t).$$
			So, equating coefficients,
				$$a_{n+1} = \int t^n d\mu(t).$$
			Since the support of $\mu$ is contained within the unit interval,
				$$|a_{k}| \leq \int |t|^{k-1} d\mu(t) \leq \int 1 d\mu(t) = a_1.$$
			This concludes the proof.
		\end{proof}
	\subsection{An extension of Agler, Tully-Doyle and Young}
		The Nevanlinna representation for Pick functions has recently been extended in \cite{aty11}. 
		\begin{thm}[Agler, Tully-Doyle, Young]
			Let $h:\Pi^2 \rightarrow \overline{\Pi}$ be analytic, furthermore suppose
				$$\lim_{s\rightarrow\infty} sh(is,is)$$
			is finite. Then $h$ has a Type I Nevanlinna representation. That is, there is a separable Hilbert space $\mathcal{H}$, with
			an unbounded self-adjoint operator $A$, a contraction $0\leq Y\leq 1$ and an $\alpha \in \mathcal{H}$ such that
				$$h(z) = <\res\alpha,\alpha>$$
			where $z_Y = Yz_1 + (1-Y)z_2.$
		\end{thm}
		Superficially, this representation looks much different than
		the Nevanlinna representation above, but its algebraic properties are essentially the same when the functional calculus
		is applied correctly. Similar representations exist for dimensions greater that $2$ for all functions in the L\"owner class \cite{amy11c}.
		However, we obtain a two varaible result and lift it to the entire Pick class, a strictly larger class than the L\"owner functions in more than $3$ variables.
		It will be essential to our stem result, the local Julia inequality for a point given in Theorem \ref{juliapoint}.
	
\section{Regulation estimates}\label{regest}
	This section gives some estimates on some specific regulators defined in Definition \ref{regdef}. Thus, for these sets
	we give an approximation of the conclusion of the above wedge-of-the-edge theorem, Theorem \ref{mainwoe}.
	
	We define special wedge-shaped sets for which it will be tractible to compute estimates on the regulators.
	\begin{defin} \label{regdef}
		Suppose there is a point $p\in \rn,$
		two positive slopes $0 < m_1 < m_2,$
		and a positive $\Dx \in \rplus$. We define the following objects
		\begin{itemize}
			\item To define the wedge, we first need its vertices,
				$$\vec{M}_n = \{m\in \rn| m_1=1, m_i \in \{m_1,m_2\}, \forall i>1\}.$$
			\item The \emph{real wedge} is the set is the union of two pyramids,
				$$\W = \text{Hull }\{p, p + \Dx\vec{M}\}
				\cup \text{Hull }\{p, p - \Dx\vec{M}\}.$$
			\item The \emph{homogenous polynomial regulator for the wedge} is the function
				$$\q{k}[z] =q^{\W}(p)[z] $$
			\item The \emph{regulated set for a wedge} is
				$$\Qh = \hat{Q}^{\W}(p)$$
			\item The \emph{real regulated set for a wedge} is
				$$\Q = Q^{\W}(p).$$
		\end{itemize}
	\end{defin}
	We show that the real regulated set for a wedge contains a parallelogram around $p.$
	We will need the fact that a theorem in \cite{gau74} combined with a calculation in \cite{sze59}
	gives an estimate for the norm of the inverse
	of the Vandermonde matrix for Chebychev interpolation on [-1,1]. That is,
		$$\|V^{-1}_n\|_\infty\leq \frac{3^{3/4}}{4}[(1+\sqrt(2))^n+(1-\sqrt(2))^n]$$
	We use this to bound a polynomial, the homogenous polynomial regulator for a wedge, by interpolating it.
	\begin{prop}\label{qpoly}
		The function $\qtwo{k}$ satisfies
		\begin{flalign*}
			& \qtwo{k}[z]\leq &
		\end{flalign*}	
		\begin{flalign*}
			& &\frac{1}{\delta^k}(k+1)\mxthm\left(|z_1|^k,\left|\frac{z_2 - (m_1+m_2)z_1}{m_2-m_1}\right|^k\right)\frac{3^{3/4}}{4}[(1+\sqrt2)^k+(1-\sqrt2)^k].
		\end{flalign*}
	\end{prop}
	\begin{proof}
		Let $p$ be a homogenous polynomial of degree $k$ with real coefficients.
		Suppose $p(\delta,\delta m)\leq \mx(\delta,\delta m)$
		for $m\in [m_1,m_2].$
		Define $r(z) = p(\delta,\delta [m_1(1-z)+m_2(1+z)]/2).$
		Let $M = \mx(\delta, \delta m_2).$
		Now $|r(z)|<M$ on $[-1,1].$
		So $$r(z)=<(z^i)^k_{i=0},V_k^{-1}r(x_i)^k_{i=0}>$$
		where $x_i$ are the Chebychev nodes.
		Thus,
			$$|r(z)|\leq (k+1)\mx(1,|z|^k)\frac{3^{3/4}}{4}[(1+\sqrt(2))^k+(1-\sqrt(2))^k].$$
		This implies 
			$$|p(\delta,\delta [m_1(z+1)+m_2(z-1)])|\leq (k+1)\mx(1,|z|^k)\frac{3^{3/4}}{4}[(1+\sqrt2)^k+(1-\sqrt2)^k].$$
		Applying homogeneity,
			\begin{flalign*}
				& |p(z_1,z_1 [m_1(z+1)+m_2(z-1)])|\leq &
			\end{flalign*}
			\begin{flalign*}
				& &\frac{z_1^k}{\delta^k}(k+1)\mx(1,|z|^k)\frac{3^{3/4}}{4}[(1+\sqrt2)^k+(1-\sqrt2)^k].
			\end{flalign*}
		Let $z=\frac{z_2/z_1 - (m_1+m_2)}{m_2-m_1}.$ So,
		\begin{flalign*}
				&|p(z_1,z_2)|\leq &
		\end{flalign*}
		\begin{flalign*}
			& & \frac{z_1^k}{\delta^k}(k+1)\mx\left(1,\left|\frac{z_2/z_1 - (m_1+m_2)}{m_2-m_1}\right|^k\right)\frac{3^{3/4}}{4}[(1+\sqrt2)^k+(1-\sqrt2)^k].
		\end{flalign*}
		Simplify to obtain the estimate,
		\begin{flalign*}
			& |p(z_1,z_2)|\leq &
		\end{flalign*}
		\begin{flalign*}
			& & \frac{1}{\delta^k}(k+1)\mx\left(|z_1|^k,\left|\frac{z_2 - (m_1+m_2)z_1}{m_2-m_1}\right|^k\right)\frac{3^{3/4}}{4}[(1+\sqrt2)^k+(1-\sqrt2)^k].
		\end{flalign*}
	\end{proof}
	This immediately implies a much simpler qualitative fact, an approximation of Theorem \ref{mainwoe} via the $n$-th root test.
	\begin{prop}
		The set $\Qhtwo$ contains a neighborhood of $p.$
		In fact, all $z$ satisfying
			$$\mxthm\left(|z_1-x|,\left|\frac{(z_2-y) - (m_1+m_2)(z_1-x)}{m_2-m_1}\right|\right)\leq \frac{\delta}{1+\sqrt2}$$
		are in $\Qhtwo.$
	\end{prop}
	We asserted in the intoduction that Theorem \ref{mainwoe} will imply that
	if $U = \{(x,y)\in\rtwo | xy >  -1\},$ then any locally operator monotone function on $U$ extends to all of $\rtwo.$
	This can be seen as a direct consequence of the above proposition, by taking $(x,y)=(0,0)$ any fixed $0 < m_1< m_2$ and letting $\delta$ tend to infinity.

	Note that the estimates in Proposition \ref{qpoly} are derived from interpolation theory.
	Better estimates would be obtained by developing a polynomial extrapolation
	theory to handle this specific problem, and this is why we defer to abstract
	regulators as opposed to the numerical estimates in the propositions above.
	
	In light of this, for $n$ variables we simply sketch that $\Qh$ has some interior.
	\begin{prop}\label{qnhood}
		The set $\Qh$ contains a neighborhood of $p.$
	\end{prop}
	\begin{proof}
		This can be seen if we take the multivariate interpolating matrix
		$ \bigotimes_{n-1} V_k $ and repeating the process for Theorem \ref{qpoly}.
	\end{proof}

\section{Local Julia inequalities}\label{juliasection}
	The Julia inequality was discovered by G. Julia in 1920 as an extension of the Schwarz lemma in \cite{ju20}.
	In one form it states that
	if $\varphi$
		extends to $\tau\in\partial \mathbb{D},$ with $|\varphi(\tau)|=1$ and
		$\varphi'(\tau)$ exists, then the following limit exists nontangentially
			$$\alpha := \lim_{\lambda\rightarrow\tau} \frac{1-|\varphi(\lambda)|}{1-|\lambda|}$$
		and
			$$\frac{|\varphi(\lambda)-\varphi(\tau)|^2}{1-|\varphi(\lambda)|^2}
			\leq \alpha\frac{|\lambda-\tau|^2}{1-|\lambda|^2}.$$
	
	The Julia inequality has been generalized by many authors to several variables. A version given in \cite{amy11a}
	states that if $\varphi:\mathbb{D}^2\rightarrow\mathbb{D}$
		extends to $\tau\in\partial \mathbb{D},$ with $|\varphi(\tau)|=1$ and
		$\varphi'(\tau)$ exists, then the following limit exists nontangentially
			$$\alpha := \lim_{\lambda\rightarrow\tau} \frac{1-|\varphi(\lambda)|}{1-\|\lambda\|}$$
		and
			$$\frac{|\varphi(\lambda)-\varphi(\tau)|^2}{1-|\varphi(\lambda)|^2}
			\leq \alpha\frac{\|\lambda-\tau\|^2}{1-\|\lambda\|^2},$$
		furthermore,
			$$\|\phi'(\tau)\|\leq\alpha.$$
	We prove the last inequality, the inequality for a derivative,
	from the bidisk on $\Pi^2$ instead, via the Nevanlinna representation in one and two variables,
	and then lift it to $\Pi^n$ using some geometry.
	This is essentially Theorem \ref{juliapoint}. 
	We then will show how the inequalities strengthen under more rigid regularity conditions than being extremal at a point. Specifically, on a line segment
	and on a general set. We call these these inequalities of directional derivatives the \emph{local Julia inequalities.}
	
	In this section we identify $1 = (1,1,\ldots,1).$
	 
	First we prove an inequality at a point, similar to the Julia-Caratheodory theorem itself.
	\begin{thm} \label{juliapoint}
		Let $p\in \rn.$ If $h\in\p(p),$ and, for some $\x\in \rnplus,$ $h'(p)[\x]$ exists,
		then, $h'(p)[1]$ exists and
			$$|h'(p)[\x]| \leq \ni{\x}h'(\x)[1].$$
	\end{thm}
	To prove this fact in general, we first prove it in two variables, so that we can later lift it to
	several variables via a geometric argument.
	\begin{lem}\label{juliapoint2}
		Let $p\in\rtwo$ If $h\in\p(p),$ and, for some $b\in \rtwoplus,$ $h'(p)[b]$ exists,
		then, $h'(p)[1]$ exists and
			$$|h'(p)[b]| \leq \ni{b}h'(p)[1].$$
	\end{lem}
	\begin{proof}
		In this proof $\frac{1}{z} = (\frac{1}{z_1},\frac{1}{z_2}).$
		We will use the Nevanlinna representation in two variables
		to calculate the derivative at $1$.
		Define $\hat{h}(z_1,z_2) = h(b_1z_1+p_1,b_2z_2+p_2).$
		Note $\hat{h}'(0)[1] = h'(p)[b].$
		Note since $\hat{h}'(0)[1]$ exists,
		$\lim_{s \rightarrow 0}\frac{1}{is}\hat{h}(is)$ exists,
		and thus $h(-\frac{1}{z})$ has a two variable Nevanlinna representation
			$$\hat{h}(-\frac{1}{z})=<(A-z_Y)^{-1}\alpha,\alpha>.$$
		Rearranging, we get
			$$\hat{h}(z)=<\left(A+\left[\frac{1}{z}\right]_Y\right)^{-1}\alpha,\alpha>.$$
		So, let $a\in \rtwoplus$ toward computing the derivative,
			$$\frac{1}{is}\hat{h}(isa)=<\left(A+\left[\frac{1}{a}\right]_Y\right)^{-1}\alpha,\alpha>.$$
		Since $\frac{1}{a}_Y$ is positive, $[\frac{1}{a}]^{1/2}_Y$ exists.
			$$\frac{1}{is}\hat{h}(isa)=
			<\left(is\left[\frac{1}{a}\right]^{-1/2}_Y
			A
			\left[\frac{1}{a}\right]^{-1/2}_Y  +  
			1\right)^{-1}
			\left[\frac{1}{a}\right]^{-1/2}_Y\alpha,
			\left[\frac{1}{a}\right]^{-1/2}_Y\alpha>.$$
		Apply the spectral theorem to $\left[\frac{1}{a}\right]^{-1/2}_YA\left[\frac{1}{a}\right]^{-1/2}_Y$
		to obtain
			$$\frac{1}{is}\hat{h}(isa)=\int \frac{1}{isx+1} d\mu(x) = \int \frac{1-isx}{1+s^2x^2} d\mu(x),$$
		where the total variation of $\mu$ is $\|\left[\frac{1}{a}\right]^{-1/2}_Y\alpha\|^2.$
		Taking the limit $s\rightarrow 0,$ we acheive
			$\hat{h}'(0)[a] = \|\left[\frac{1}{a}\right]^{-1/2}_Y\alpha\|^2.$
		So $h'(p)[b_1a_1,b_2a_2] = \|\left[\frac{1}{a}\right]^{-1/2}_Y\alpha\|^2.$
		Thus,
		$h'(p)[b] = \|\alpha\|^2$ and
		$h'(p)[1,1] = \|\left[\frac{1}{b}\right]^{-1/2}_Y\alpha\|^2.$
		Note, $\ni{b} \left[\frac{1}{b}\right]^{-1}_Y \geq 1$
		Thus, $|h'(p)[b]| \leq h'(p)[1].$
	\end{proof}
	Now we can prove the full result using some geometry.
	\begin{proof}[Proof of Theorem \ref{juliapoint}]
		If $\x$ is a multiple of $1,$ there is nothing to prove. Suppose not.
		Define
			$$\vec{v}_1 = \frac{\ni{\frac{1}{\x}}\ni{x} -
			\ni{\frac{1}{\x}}\x}{\ni{\frac{1}{\x}}\ni{x}-1}, 
			\vec{v}_2 = \frac{\ni{\frac{1}{\x}}\x - 1}{\ni{\frac{1}{\x}}\ni{x}-1}.$$
		Note, $\vec{v}_1, \vec{v}_2 \in \rnplus$, $\vec{v}_1 + \vec{v}_2 = 1,$ and
		$\frac{1}{\ni{\frac{1}{\x}}}v_1 + \ni{\x}v_2= \x.$
		Thus $f(\omega_1,\omega_2)=h(p + \omega_1 \vec{v}_1+\omega_2 \vec{v}_2)$ is a Pick function
		of two variables and satisfies \ref{juliapoint2},
			$$|f'(0)[b]|\leq \ni{b}f'(0)[1].$$
		Note that $\frac{1}{\ni{\frac{1}{\x}}}\leq\ni{\x}.$ Thus,
			$$|h'(p)[\x]|\leq \ni{x}h'(p)[1].$$
	\end{proof}

	We now strengthen the domain to an entire line segment and obtain a stronger result, a result for
	higher order directional derivatives.
	\begin{thm} \label{julialine}
		Let $L$ be a line segment in $\rn$ with positive slope. Denote
		$p$ as the midpoint of $L$ and $\x$ as the difference between the right
		endpoint and the midpoint $b$. If $h\in\p(L)$, then
			$$|h^{(k)}(p)[\x]| \leq k!\ni{\x}h'(p)[1]$$
	\end{thm}
	\begin{proof}
		Suppose $L$, $p$ and $\x$ are as in the statement of the theorem.
		Consider the function $f(w) = h(w\x+p)-h(p).$
		By Lemma \ref{juliaone}, $f$ is analytic and its power series
			$$f(w)=\sum a_kw^k$$
		satisfies $|a_k| \leq a_1.$
		Note $k!a_k = h^{(k)}(p)[\x].$
		Thus, applying the above and Theorem \ref{juliapoint},
			$$|h^{(k)}(p)[\x]|\leq k!|h'(p)[\x]| \leq k!\ni{x}h'(p)[1].$$
	\end{proof}
	Now we can prove Theorem \ref{liouville}.
	\begin{proof}
		Suppose $f \in L(\mathbb{R}^n).$ By Theorem \ref{julialine},
			$$|f^{(k)}(0)[s\x]|\leq k!s\ni{x}f'(p)[1]$$
		for all $x \in \rnplus, s\geq 0.$
		So,
			$$|f^{(k)}(0)[\x]|\leq k!\frac{1}{s^{k-1}}\ni{x}f'(p)[1]$$
		Taking $s\rightarrow \infty,$
			$$|f^{(k)}(0)[\x]|\leq 0$$
		Thus $f$ is linear on $\rnplus,$ so by continuation,
		$f$ is linear.
	\end{proof}

	Finally, we shall obtain a local Julia inequality for points on the interior of a set. It is given in terms of the regulators.
	\begin{thm}\label{julianhood}
		Let $S\subset \rn.$ Let $p$ be on the interior of $S.$
		If $h\in \p(S),$ then
			$$|h^{(k)}(p)[\z]| \leq k!q^S_{k}(p)[\z]h'(p)[1].$$
	\end{thm}
	\begin{proof}
		By Theorem \ref{julialine},
			$$h^{(k)}(p)[\x] \leq k!d^S(p)[\x]\ni{x}h'(p)[1],$$
		for all $\x \in \rtwoplus$.
		Since $h$ is analytic at $p$ by the edge of the wedge theorem, $h^k(p)[\x]$ is a homogenous polynomial
		in the entries of $\x$. Therefore, by definition of $q^S_{k}$,
			$$|h^{(k)}(p)[\z]| \leq  k!q^S_{k}(p)[\z]h'(p)[1].$$
	\end{proof}
	
\section{Proof of the main result}\label{woesection}
	This section is devoted to proving our wedge-of-the-edge theorem. First we prove a pointwise version.
	\begin{lem}\label{wedgelem}
		Let $E$ contain a neighborhood of $p.$
		Then,
			$$\p(E)=\p(E \cup Q^E(p)).$$
	\end{lem}
	\begin{proof}
		Consider the function,
			$$H(p+\z)=\sum^{\infty}_{k=0} \frac{h^{(k)}(p)[\z]}{k!}.$$
		Applying the $k$-th root to the bound given via Theorem \ref{julianhood},
		we get that this sum converges on $\hat{Q}^E(p)$ via its defintion.
		Note $H$ agrees with $h$ on the interior of $i\{p + s\x|\x\in\rnplus,0 < s < d^S(p[\x])\},$
		and $\hat{Q}^E(p)$ contains a neighborhood of $p$ by Proposition \ref{qnhood} (since any open set contains a wedge.)
		Open subsets of $i\rn$ are sets of uniqueness for analytic functions on $\Pi^n$.
		Thus, $H$ is an analytic continuation of $h.$
		So, $$\p(E)=\p(E \cup Q^E(p)).$$
	\end{proof}
	Now, we can prove the main result in terms of the Pick class.
	\begin{thm}[A wedge-of-the-edge theorem]
		Let $S\subset \rn.$
			$$\p(S)=\p(Q^S).$$
	\end{thm}
	\begin{proof}
		By Theorem \ref{wedgelem} there is a set containing $S$, the maximal $F$ such that $\p(S)=\p(F)$ satisfies $F \supset F\cup Q^F(p)$ for every $p$ in the interior of $F.$
		Note that by definition, $Q^S$ is the minimal such set and is thus contained in $F$.
	\end{proof}
	\begin{cor}
		Let $S\subset \rn.$
			$$\L(S)=\L(Q^S).$$
	\end{cor}
	This implies the main result Theorem \ref{mainwoe} via \ref{amy11c}.

\bibliography{references}
\end{document}